\documentclass[a4paper,10pt]{amsart}
\usepackage{amsmath,amssymb,graphicx,url}

\newtheorem{theorem}{Theorem}[section]

\newtheorem{lemma}[theorem]{Lemma}

\theoremstyle{definition}

\newtheorem{prob}[theorem]{Problem}

\begin{document}
\title[]{An extended symmetric union with multiple tangle regions and its Alexander polynomial}
\author{Teruaki KITANO \and Yasuharu NAKAE}
\address{Department of Information Systems Science, Faculty of Science and Engineering, Soka University, Tangi-machi 1-236, Hachioji, Tokyo, 192-8577, Japan}
\email{kitano@soka.ac.jp}
\address{Graduate School of Engineering Science, Akita University,
1-1 Tegata Gakuen-machi, Akita city, Akita, 010-8502, Japan.}
\email{nakae@math.akita-u.ac.jp}

\begin{abstract}
The authors recently introduced a new construction of a knot as an extended symmetric union of a knot with a single tangle region.
In this paper, we generalize the construction to include multiple tangle regions.
The constructed knot $K$ with a partial knot $\hat{K}$ and multiple tangle regions satisfies the following two properties: its Alexander polynomial is the product of the Alexander polynomials of the numerators of these tangles and the square of the Alexander polynomial of the partial knot $\hat{K}$, and there exists a surjective homomorphism from the knot group of $K$ to that of $\hat{K}$
which maps the longitude of $K$ to the trivial element.
\end{abstract}

\maketitle

\setlength{\baselineskip}{14pt}

\section{Introduction}
Let $K$ be a knot in $S^3$, and let $G(K)=\pi_1(S^3\setminus K)$ denote the knot group of $K$.
For prime knots $K_1$ and $K_2$, we write $K_1 \geq K_2$ if there exists an epimorphism $\varphi: G(K_1)\to G(K_2)$ between these knot groups and $\varphi$ preserves the meridian.
The relation $\geq$ becomes a partial order on the set of prime knots, see \cite{KS2008}.
The pairs of knots satisfying $K_1 \geq K_2$ are studied by Kitano and Suzuki in \cite{KS2005, KS2005C} and by Horie, Kitano, Matsumoto, and Suzuki in \cite{HKMS2011, HKMS2011E}.
A pair of knots with $K_1 \geq K_2$ has some interesting properties regarding the relationship between these knots.
For instance, when $K_1 \geq K_2$, the Alexander polynomial of $K_1$ is divisible by that of $K_2$; if $K_1$ is a fibered knot, then $K_2$ is also fibered.

The concept of a symmetric union of a knot was introduced by Kinoshita and Terasaka in \cite{KT1957}, and was generalized by Lamm in \cite{L2000}.
Recently, the authors proposed a new construction that extends the symmetric union of a knot and showed that the resulting knots have some properties similar to those of the original symmetric union.
In that paper, the twist region of a symmetric union of a knot with a single twist region is replaced by a tangle region \cite{KN2025}.
In this paper, we extend the construction to include multiple tangle regions.

Let $D$ be a planar diagram of a knot $\hat{K}$ in $S^3$.
We consider the diagram $D^\ast$ as the diagram of the knot obtained by reflecting $\hat{K}$ across the plane orthogonal to the plane on which the diagram $D$ lies.  
Let $T_i$ be tangles such that the left-side endpoints are connected within $T_i$, and the right-side endpoints are also connected within $T_i$. We call this type of tangle an ``even type'' tangle.
We take $n+1$ arcs $\tilde{x}_0, \dots , \tilde{x}_{n-1}, \tilde{x}_n$ from the diagram $D$ as shown on the left of Figure \ref{maintheorem_diagram}.
Let $\tilde{x}_0^\ast, \dots, \tilde{x}_{n-1}^\ast, \tilde{x}_n^\ast$ be arcs in $D^\ast$ corresponding to $\tilde{x}_0, \dots, \tilde{x}_{n-1}, \tilde{x}_n$.
Placing the diagrams $D$ and $D^\ast$ symmetrically, forming a connected sum of $D$ and $D^\ast$ by using $\tilde{x}_0$ and $\tilde{x}_0^\ast$, and replacing each pair of arcs $\tilde{x}_i$ and $\tilde{x}_i^\ast$ with the tangle $T_i$,
we define a knot $K$ as consisting of tangle regions $\{T_i\}$ along with the diagrams $D$ and $D^\ast$ as shown on the right of Figure \ref{maintheorem_diagram}.

For a tangle $T$ with four endpoints arranged at the corners,
the numerator $N(T)$ of $T$ is formed by connecting the top endpoints to each other
and the bottom endpoints as well.
A simple closed curve on the boundary of the tubular neighborhood of a knot is referred to as the preferred longitude if it bounds a Seifert surface.
In this paper, a longitude refers to a preferred longitude.
For a knot $K$, $\Delta_K(t)$ denotes the Alexander polynomial of $K$.

\begin{theorem}\label{maintheorem}

The knot $K$, constructed as described above, satisfies the following properties.
\begin{enumerate}
\item $\Delta_K(t)=\Delta_{N(T_1)}(t)\cdots\Delta_{N(T_n)}(t)\left(\Delta_{\hat{K}}(t)\right)^2$.
\item There is an epimorphism $\varphi:G(K)\to G(\hat{K})$ such that
$\varphi$ maps a meridian of $K$ to a meridian of $\hat{K}$,
and maps the preferred longitude of $K$ to the trivial element in $G(\hat{K})$.
\end{enumerate}
\end{theorem}

\begin{figure}[h]
	\centerline{\includegraphics[keepaspectratio]{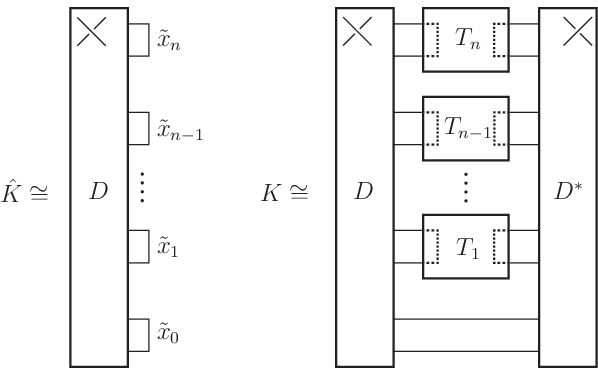}}
\caption{the diagram of a knot $K$}
\label{maintheorem_diagram}
\end{figure}

In Theorem \ref{maintheorem}, when all $T_i$ consist of an even number of twists, property (1) follows from a theorem by Kinoshita and Terasaka \cite[Theorem 2]{KT1957} and Lamm \cite[Theorem 2.4]{L2000}, property (2) follows from a theorem by Eisermann \cite[Theorem 3.3]{L2000}.
Therefore, the construction in Theorem \ref{maintheorem} extends a symmetric union of a knot where all twist regions are of even type, and these knots inherit some properties of this type of symmetric union.

For a symmetric union of a knot with only even twist regions, 
Boileau, Kitano, and Nozaki showed that a similar property related to (1) of Theorem \ref{maintheorem} holds for twisted Alexander polynomial \cite[Theorem 1.6]{BKN2024}.
However, for the case of ``even type'' tangles $T_i$, it may be challenging to prove the property (1) within the setting of the twisted Alexander polynomials.

As a corollary, similar to \cite[Corollary 1.5]{KN2025}, we can easily construct a non-fibered knot $K$ whose knot group surjectively maps to that of a fibered knot $\hat{K}$, and the longitude of $K$ is mapped to a trivial element by applying the construction from Theorem \ref{maintheorem} as follows.
Let $\hat{K}$ be a fibered knot; for example, we could take $\hat{K}$ to be the figure-eight knot. Also, let $T_k$ be a tangle for which the Alexander polynomial of numerator $N(T_k)$ is not monic.
Then we see that the Alexander polynomial of $K$, constructed using the method of Theorem \ref{maintheorem} and including the tangle $T_k$ in the set of tangles $\{T_i\}$, is not monic.
Therefore, the constructed knot $K$ is not fibered since a fibered knot has a monic Alexander polynomial.

\section{Proof of Main Theorem}
This section is divided into two subsections for the proofs of (1) and (2) of Theorem \ref{maintheorem}.

\subsection{Proof of Theorem \ref{maintheorem} (1)}
Let $T(p/q)$ be a rational tangle with a slope $p/q$.
The two special rational tangles, the $0$-slope tangle $T(0/1)$ and the infinite slope tangle $T(1/0)$, are illustrated in Figure \ref{specialtangles}.

\begin{figure}[h]
	\centerline{\includegraphics[keepaspectratio]{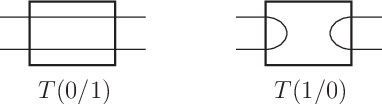}}
\caption{two special rational tangles}
\label{specialtangles}
\end{figure}

To prove Theorem \ref{maintheorem}, we need the following lemma, which is essentially the same as Theorem 2.3 of \cite{L2000}.

\begin{lemma}\label{numeratorofT0}
Let $L_0^{(n)}$ be a link in which all tangles $T_i$ of $K$ are replaced by a $0$-slope tangle $T(0/1)$.
Then $\Delta_{L_0^{(n)}}(t)=0$.
\end{lemma}
\begin{proof}
The proof of the lemma is almost analogous to the proof of Theorem 2.3 of \cite{L2000}.
In this proof, $L$ denotes the link $L_0^{(n)}$.
In order to prove the lemma, we will show that $\Delta_{L}(t)=0$ by using the Alexander matrix of $L$.

Taking an orientation of $\hat{K}$, we assign the orientation of the diagram $D$ induced from $\hat{K}$, and assign the reversed orientation for $D^\ast$ from the orientation of $D$.
We can then assign the orientation to $K$, also to $L$.
As shown in Figure \ref{labelsoflemma}, we label the regions divided by $L$ in the diagram and the crossings of $L$ as follows.
The region $a_1$ is the external region,
$a_0$ is the region surrounded by the arcs that make a connected sum of $D$ and $D^\ast$.
For $i=1,2,\dots,n$, $\bar{b}_i$ is the region formed by replacing the tangle $T_i$ with a tangle $T(0/1)$,
and $b_i$ is the region below $\bar{b}_i$.
Let $Q$ and $Q^\ast$ be square domains corresponding to the diagrams $D$ and $D^\ast$.
For $i=1,2,\dots,m$, $c_i$ is a region divided by $L$ in $Q$ and $c_i^\ast$ is the corresponding region in $Q^\ast$.
For $i=1,2,\dots,\ell$, $d_i$ is a crossing point in $Q$ and $d_i^\ast$ is the corresponding crossing point in $Q^\ast$. \begin{figure}[h]
	\centerline{\includegraphics[keepaspectratio]{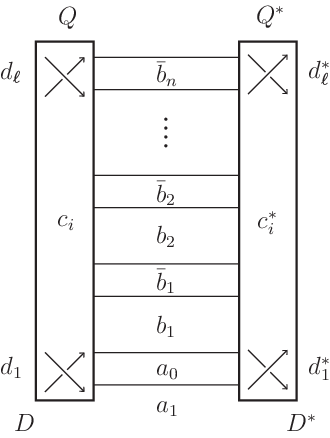}}
\caption{labels in the diagram of $L$}
\label{labelsoflemma}
\end{figure}

If there is a region that does not touch any of the crossings, then a trivial link component exists, and consequently $\Delta_L(t)=0$.
Therefore, we assume that all regions touch some crossings, which implies that each edge in the diagram of $L$ connects two crossing points, and each crossing point has four edges.

If the number of regions is $m$ and the number of crossings is $\ell$ in $Q$,
by regarding that the diagram of $L$ lies on the 2-sphere $S^2$,
we see that $\ell=m+n$ by calculating the Euler characteristic. 

To calculate the Alexander matrix of $L$,
four labels $x$, $-x$, $1$, and $-1$ are assigned to the four regions around a crossing point $d_i$ and $d_i^\ast$ as shown in Figure \ref{labelsoflemma02}.
Figure \ref{labelsoflemma02} also shows the correspondence between the labels of regions $c_i$ and $c_i^\ast$ around the crossing points $d_i$ and $d_i^\ast$, respectively.

\begin{figure}[h]
	\centerline{\includegraphics[keepaspectratio]{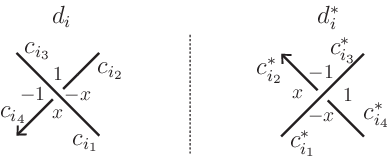}}
\caption{labels aroud crossing $d_i$ and $d_i^\ast$}
\label{labelsoflemma02}
\end{figure}

Note that the labels for the regions $(c_{i_1}, c_{i_2}, c_{i_3}, c_{i_4})=(x, -x, 1, -1)$ around the crossing point $d_i$ correspond to the $(-1)$-times of labels $(c_{i_1}^\ast, c_{i_2}^\ast, c_{i_3}^\ast, c_{i_4}^\ast)=(-x, x, -1, 1)$ around $d_i^\ast$.

The Alexander matrix $A(x)$ of $L$ consists of the elements $\{x, -x, 1, -1\}$, where the element in the $p$-th row and $q$-th column is an element labeled according to the rule stated above for the $p$-th crossing point and the $q$-th region.
Let $\hat{A}(x)$ be a matrix from which two columns corresponding to two adjacent regions have been removed.
Then we obtain the Alexander polynomial of $L$ by $\Delta_L(x)=\det \hat{A}(x)$.

Following the labels assigned to the diagram of $L$, we derive the Alexander matrix $A(x)$ as follows:
\begin{equation*}
A(x)=
 \left(
\begin{array}{c|c|c|c}
 \begin{array}{cc} \ast & \ast \\ \vdots & \vdots \\ \ast & \ast \end{array} & N & M & O \\
 \hline
 \begin{array}{cc} \ast & \ast \\ \vdots & \vdots \\ \ast & \ast \end{array} & -N & O & -M
\end{array}
\right)
\end{equation*}

For the matrix $A(x)$ described above, the rows correspond to crossing points $d_1$, $\dots$, $d_\ell$, $d_1^\ast$, $\dots$, $d_\ell^\ast$,
and the columns correspond to regions $a_0$, $a_1$, $b_1$, $\bar{b}_1$, $\dots$, $b_n$, $\bar{b}_n$, $c_1$, $\dots$, $c_m$, $c_1^\ast$, $\dots$, $c_m^\ast$.
The elements marked $\ast$ are not related to the calculation of the Alexander polynomial.
The blocks $N$ and $-N$, as well as $M$ and $-M$ are obtained through the symmetry of $D$ and $D^\ast$ as noted before.
By removing the two columns corresponding to the regions $a_0$ and $a_1$ from $A(x)$, we obtain the square matrix $\hat{A}(x)$ with size $2(m+n)$.

We can modify $\hat{A}(x)$ to the simple matrix $\tilde{A}(x)$ by adding the first $\ell=m+n$ rows to the last $\ell$ rows, and then adding the last $m$ columns to the middle $m$ columns as follows.

\begin{equation*}
\hat{A}(x)=
\left(
\renewcommand{\arraystretch}{1.5}
\begin{array}{c|c|c}
N & M & O \\
\hline
-N & O & -M
\end{array}\right)
\rightarrow
\left(
\renewcommand{\arraystretch}{1.5}
\begin{array}{c|c|c}
N & M & O \\
\hline
O & M & -M
\end{array}\right)
\rightarrow
\left(
\renewcommand{\arraystretch}{1.5}
\begin{array}{c|c|c}
N & M & O \\
\hline
O & O & -M
\end{array}\right)
=\tilde{A}(x)
\end{equation*}

The last $m+n$ rows contain an $(m+n)\times m$ block $-M$ with non-zero elements.
The rank of the matrix formed by these last $m+n$ rows of $\tilde{A}(x)$ is at most $m$ since all entries in the first $2n+m$ columns are zero.
Therefore, the row vectors of these last $m+n$ rows, including the block $-M$, are linearly dependent,
and the determinant of $\tilde{A}(x)$ equals zero.

As a result, we can conclude that $\Delta_L(x)=\det \hat{A}(x)=0$.
 
\end{proof}

\begin{proof}[Proof of Theorem \ref{maintheorem}(1)]
We shall prove Theorem\ref{maintheorem}(1) by induction on the number $n$ of tangle regions.
Let $K_n$ be the knot with $n$ tangle regions described in Theorem \ref{maintheorem}.

In the case of $n=1$, the diagram of $K_1$ is regarded as the one shown in Figure \ref{pf1firstcase},
where $\tilde{T}_1$ is the tangle obtained by rotating $T_1$ by $\pi$.
Let $T_0$ be the tangle formed by connecting two diagrams $D$ and $D^\ast$.

\begin{figure}[h]
	\centerline{\includegraphics[keepaspectratio]{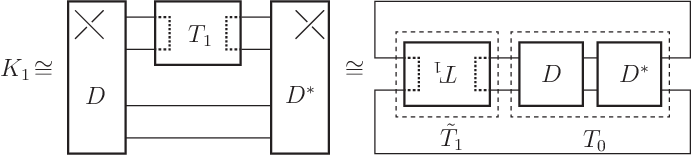}}
\caption{the diagram of $K_1$ in the case $n=1$}
\label{pf1firstcase}
\end{figure}

The tangle sum of two tangles $\tilde{T}_1$ and $T_0$ denotes $\tilde{T}_1+T_0$,
and we observe that $K_1=N(\tilde{T}_1+T_0)$ according to the notation.
$\nabla_K(z)$ denotes the Conway polynomial of $K$,
and $D(T)$ denotes the denominator of the tangle $T$
which is constructed by connecting the left-side endpoints of $T$ to each other and by also connecting the right-side endpoints.
Applying Conway's fraction formula (see, \cite[Theorem 7.9.1]{Cr2004}, \cite[Theorem 2.3]{Ka1981}) to $\nabla_{N(\tilde{T}_1+T_0)}$,
we obtain the following formula;
\begin{align*}
\nabla_{N(\tilde{T}_1+T_0)}(z)
=\nabla_{N(\tilde{T}_1)}(z)\nabla_{D(T_0)}(z)+\nabla_{D(\tilde{T}_1)}(z)\nabla_{N(T_0)}(z).
\end{align*}

$N(T_0)$ is equivalent to the $L_0^{(1)}$ that appears in Lemma \ref{numeratorofT0} for the case $n=1$.
Then $\nabla_{N(T_0)}(z)=0$ because $\Delta_{N(T_0)}(t)=0$.
Furthermore, 
since $D(T_0)$ is equivalent to the connected sum of two diagrams $D$ and $D^\ast$,
we see that $\nabla_{D(T_0)}(z)=\left(\nabla_{\hat{K}}(z)\right)^2$.
By the above argument and the fact that $N(\tilde{T}_1)$ is equivalent to $N(T_1)$, we can see the following;
\begin{align*}
\nabla_{K_1}(z)=\nabla_{N(\tilde{T}_1+T_0)}(z)=\nabla_{N(T_1)}(z)\left(\nabla_{\hat{K}}(z)\right)^2,
\end{align*}
and we conclude that $\Delta_{K_1}(t)=\Delta_{N(T_1)}(t)\left(\Delta_{\hat{K}}(t)\right)^2$.
This completes the proof for the case of $n=1$.

We next assume that $K_{n-1}$ satisfies $\Delta_{K_{n-1}}(t)=\Delta_{N(T_1)}(t)\cdots\Delta_{N(T_{n-1})}(t)\left(\Delta_{\hat{K}}(t)\right)^2$,
which is equivalent to $\nabla_{K_{n-1}}(z)=\nabla_{N(T_1)}(z)\cdots\nabla_{N(T_{n-1})}(z)\left(\nabla_{\hat{K}}(z)\right)^2$.
We will prove that $\nabla_{K_n}(z)=\nabla_{N(T_n)}(z)\nabla_{K_{n-1}}(z)$.
The proof described below is an analog of Conway's fraction formula proof.

We construct the skein tree that relates to all the crossings in the tangle $T_n$.
If we replace $T_n$ with $T(1/0)$, we obtain $K_{n-1}$.
Let $\bar{K}_{n-1}$ be the knot formed by replacing $T_n$ with $T(0/1)$.
All branches of the skein tree, except those ending at $K_{n-1}$ and $\bar{K}_{n-1}$, contain trivial components.
Therefore, these branches can be ignored when calculating the Conway polynomial.
Let $f(z)$ be the polynomial formed by the products of $\pm z$ or $1$ along the paths from the root of the tree to the endpoints with $K_{n-1}$,
and let $g(z)$ be the polynomial derived from the paths to the endpoints with $\bar{K}_{n-1}$. 
Using the polynomials $f(z)$ and $g(z)$,
we can conclude that
$$\nabla_K(z)=f(z)\nabla_{K_{n-1}}(z)+g(z)\nabla_{\bar{K}_{n-1}}(z).$$
Following the procedure used in constructing the skein tree for $K_n$,
we also build the skein tree related to all the crossings in the numerator $N(T_n)$ of $T_n$.
The endpoints of the tree that do not contain a trivial link component in their branches are $N(T(1/0))$ and $N(T(0/1))$.
Since $\nabla_{N(T(1/0))}(z)=1$ and $\nabla_{N(T(0/1))}(z)=0$ we can conclude that
$$\nabla_{N(T_n)}(z)=f(z)\nabla_{N(T(1/0))}(z)+g(z)\nabla_{N(T(0/1))}(z)=f(z).$$

We analyze the skein tree for $\bar{K}_{n-1}$, which relates to all the crossings in the $T_{n-1}$.
As in the previous case, the endpoints of branches without trivial link components include the tangles $T(1/0)$ and $T(0/1)$.
These types are similar to $\bar{K}_{n-1}$, meaning they contain more than two tangles of type $T(0/1)$.
Repeating this procedure, we see that all branches of the skein tree for $\bar{K}_{n-1}$, which relate to all the crossings in the tangle regions $T_{n-1}, \dots T_1$, have endpoints with the type of  $L_0^{(k)}$ for some $k$ in Lemma \ref{numeratorofT0}.
Therefore we conclude $\nabla_{\bar{K}_{n-1}}(z)=0$ since $\nabla_{L_0^{(k)}}(z)=0$ for any $k$.

As a result, we observe that $\nabla_{K}(z)=f(z)\nabla_{K_{n-1}}(z)=\nabla_{N(T_n)}(z)\nabla_{K_{n-1}}(z)$.
By the assumption of the induction, we finally conclude
$$\nabla_{K}(z)=\nabla_{N(T_n)}(z)\nabla_{N(T_{n-1})}(z)\cdots\nabla_{N(T_1)}(z)\left(\nabla_{\hat{K}}(z)\right)^2.$$

\end{proof}

\subsection{Proof of Theorem \ref{maintheorem} (2)}

The proof of Theorem \ref{maintheorem} (2) extends the proof of \cite[Theorem 1.1 (2)]{KN2025},
and it is also analogous to the proof of \cite[Theorem 2.3]{L2000} as noted in \cite{KN2025}.
Although the method used to prove Theorem \ref{maintheorem} (2) is essentially the same as in Theorem 1.1 (2) of \cite{KN2025}, we will present the proof in a self-contained manner in this paper.
To explain the proof explicitly in the case, we will assign labels to the arcs in the diagram of $K$ as follows.

We assign the orientation of $K$ by fixing an orientation of the partial knot $\hat{K}$ as mentioned at the beginning of the proof of Lemma \ref{numeratorofT0}.
Note that although all the orientation of $\tilde{x}_0, \tilde{x}_1, \dots, \tilde{x}_n$ are indicated downward in Figure \ref{labelsofpf2},
it does not affect the following proof if some arcs $\tilde{x}_i$ go upward.

Let $x_1, x_2, \dots, x_m$ be arcs in $D$, and $x_1^\ast, x_2^\ast, \dots, x_m^\ast$ be the corresponding arcs in $D^\ast$.
Note that the number $m$ is unrelated to the number $m$ in the proof of Lemma \ref{numeratorofT0}.
The arcs $\tilde{x}_0$ and $\tilde{x}_0^\ast$ are used to create a connected sum between $D$ and $D^\ast$,
and also $\tilde{x}_i$ and $\tilde{x}_i^\ast$ are used to connect the tangle $T_i$ to $D$ and $D^\ast$,
where $\tilde{x}_i \in \{x_1, \dots, x_m\}$, $\tilde{x}_i^\ast \in \{x_1^\ast, \dots, x_m^\ast\}$.
After performing the connected sum, the part of the arc $\tilde{x}_0$ which is oriented outward from $D$ is transformed into the arc $\alpha_0$, and the part of $\tilde{x}_0$ which is oriented inward to $D$ is transformed into the arc $\alpha_1$.
The arc $\tilde{x}_0^\ast$ is similarly divided into $\alpha_0$ and $\alpha_1$.

\begin{figure}[h]
	\centerline{\includegraphics[keepaspectratio]{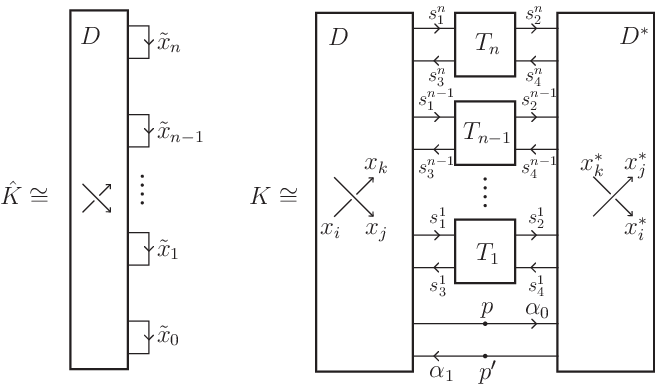}}
\caption{labels in the diagram of $K$}
\label{labelsofpf2}
\end{figure}

Let $s_1^i, s_2^i, \dots, s_{n_i}^i$ be the arcs in the tangle $T_i$,
where $n_i$ is the number of arcs in $T_i$.
Especially let $s_1^i$, $s_2^i$, $s_3^i$ and $s_4^i$ be the arcs connecting the tangle $T_i$ to $D$ and $D^\ast$.
The arc $\tilde{x}_i$ is divided and transformed into the arcs $s_1^i$ and $s_3^i$,
and $\tilde{x}_i^\ast$ is divided and transformed into the arcs $s_2^i$ and $s_4^i$,
similar to the case of $\tilde{x}_0$ and $\tilde{x}_0^\ast$.

This assigns labels to all the arcs in the diagram of $K$.
These arcs are regarded as generators of the knot group $G(K)$ in a Wirtinger presentation, using the orientation assigned above.

We will explain the process of finding the word $\gamma$ of the longitude $\lambda_K$ of $K$ in $G(K)$.

Let $p$ and $p'$ be points on the arcs $\alpha_0$ and $\alpha_1$ respectively, as shown in Figure \ref{labelsofpf2}.
Starting from the point $p$ and moving along the path following the orientation of $K$,
we read the word $w^\ast=y^\ast {(x^\ast)}^{-1}$ or $w^\ast={(y^\ast)}^{-1}x^\ast$ when we passing through the under-crossing point of $D^\ast$, similarly the word $s_j^i {(s_k^i)}^{-1}$ or ${(s_j^i)}^{-1} s_k^i$ in the tangle $T_i$,
depending on whether the crossing is positive or negative as shown in Figure \ref{labelsofpf2long}. 
Let $\gamma_+$ be the word obtained by repeating this process until reaching the point $p'$.
We also obtain the word $\gamma_-$ similarly, starting from $p$ and moving backward along the path in $D$ and all $T_i$, reading words at each under-crossing point in reverse order until reaching point $p'$. 

\begin{figure}[h]
	\centerline{\includegraphics[keepaspectratio]{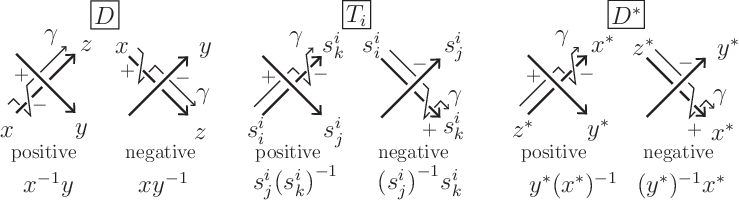}}
\caption{words for $\gamma_+$ and $\gamma_-$ on each under-crossing}
\label{labelsofpf2long}
\end{figure}

Let $\gamma$ be a concatenation of two words $\gamma_-$ and $\gamma_+$, such that $\gamma=\gamma_-\cdot\gamma_+$.
According to the previous construction, a loop represented by $\gamma$ is parallel to the knot $K$, and the linking number between the loop and $K$ is equal to zero.
Therefore, $\gamma$ represents the longitude $\lambda_K$ of $K$.

We will precisely describe the structure of sub-words in $\gamma$ as follows.
On both sides of point $p$, at the first encountered under-crossing points $p_1$ and $p_1^\ast$,
the added words $w_1$ and $w_1^\ast$ included in the diagrams $D$ and $D^\ast$ respectively are of the following types:
when the under-crossing point $p_1$ is positive, $(w_1, w_1^\ast)=(x_{i_1}^{-1}y_{i_1}, (y_{i_1}^\ast)^{-1} x_{i_1}^\ast)$,
or 
when $p_1$ is negative, $(w_1, w_1^\ast)=(x_{i_1}y_{i_1}^{-1}, y_{i_1}^\ast(x_{i_1}^\ast)^{-1})$
as shown in Figure \ref{wordofpf2long_01}.

\begin{figure}[h]
	\centerline{\includegraphics[keepaspectratio]{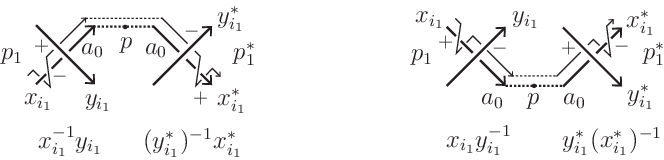}}
\caption{$(p_1, p_1^\ast)=(\text{positive},\text{negative})$ or $(\text{negative},\text{positive})$}
\label{wordofpf2long_01}
\end{figure}

At the under-crossing points $p_k$ and $p_k^\ast$, which are corresponding points in $D$ and $D^\ast$ respectively, the added words are similarly of the following types:
$(w_k, w_k^\ast)=(x_{i_k}^{-1}y_{i_k}, (y_{i_k}^\ast)^{-1} x_{k_1}^\ast)$ if $p_k$ is positive,
or 
$(w_k, w_k^\ast)=(x_{i_k}y_{i_k}^{-1}, y_{i_k}^\ast(x_{i_k}^\ast)^{-1})$ if $p_k$ is negative
as shown in Figure \ref{wordofpf2long_02}.
At the under-crossing points within the tangle $T_i$, during the reading process, the added word is either $s_j^i(s_k^i)^{-1}$ or $(s_j^i)^{-1}s_k^i$ as illustrated in the middle of Figure \ref{labelsofpf2long}.

\begin{figure}[h]
	\centerline{\includegraphics[keepaspectratio]{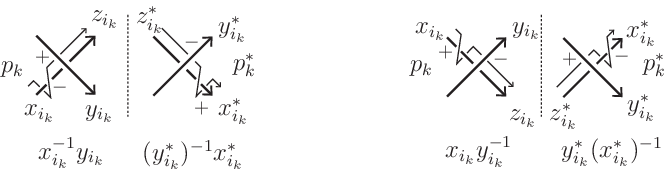}}
\caption{$(p_k, p_k^\ast)=(\text{positive},\text{negative})$ or $(\text{negative},\text{positive})$}
\label{wordofpf2long_02}
\end{figure}

As a result of the reading process,
two words, $\gamma_+$ and $\gamma_-$, take the following form;
\begin{align*}
\gamma_+
&= w_1^\ast w_2^\ast \cdots w_{i_1}^\ast S_+^{k_1} w_{i_1+1}^\ast \cdots w_{i_j}^\ast S_+^{k_j} w_{i_j+1}^\ast \cdots w_{i_n}^\ast S_+^{k_n} w_{i_n+1} \cdots w_m^\ast \\
\gamma_-
&= w_m \cdots w_{i_n+1} S_-^{k_n} w_{i_n} \cdots w_{i_j+1} S_-^{k_j} w_{i_j} \cdots w_{i_1+1} S_-^{k_1} w_{i_1} \cdots w_2 w_1
\end{align*}

where $S_+^{k_i}$ and $S_-^{k_i}$ are words created by multiplying words $s_j^i (s_k^i)^{-1}$ or $(s_j^i)^{-1}s_k^i$, which correspond to the under-crossing points in $T_i$.  
 
We next consider the knot groups $G(\hat{K})$ and $G(K)$.
The knot group $G(\hat{K})$ is generated by the arcs $\hat{x}_1, \dots, \hat{x}_m$, which correspond to the arcs $x_1, \dots, x_m$ in the diagram $D$.
The knot group $G(K)$ is also generated by the arcs $x_1, \dots, x_m, x_1^\ast, \dots, x_m^\ast$ in the diagrams $D$ and $D^\ast$, $\alpha_0$, $\alpha_1$, and $s_k^i$ in the tangles $T_i$.
We regard these generators as those of the Wirtinger presentation,
and let $r_i$ be a relator corresponding to a crossing in $D$, $r_i^\ast$ to a crossing in $D^\ast$, and $r_k^i$ to a crossing in $T_i$.
Note that the generators corresponding to the arcs $\tilde{x}_0$ and $\tilde{x}_0^\ast$ in the diagram $D$ and $D^\ast$ are devided and transformed into generators corresponding to the arcs $\alpha_0$ and $\alpha_1$,
similarly $\tilde{x}_i$ and $\tilde{x}_i^\ast$ have also been transformed into $s_k^i$, $k=1,2,3,4$.

Using the Wirtinger presentation above, we define a map $\varphi:G(K)\to G(\hat{K})$ as follows.

\begin{equation*}
\varphi:G(K)\to G(\hat{K}):\;
\left\{\begin{array}{l}
 x_i \mapsto \hat{x}_i,\; x_i^\ast \mapsto \hat{x}_i, \;\; i=1,2,\dots, m \\[5pt]
 \alpha_0 \mapsto \tilde{x}_0, \alpha_1 \mapsto \tilde{x}_0, \\[5pt]
 s_k^i \mapsto \tilde{x}_i, \;\; i=1,2,\dots, n, \;\; k=1,2,\dots,n_i
\end{array}\right.
\end{equation*}

The relator $\hat{r}_i$ in $G(\hat{K})$ is of the form $\hat{r}_i=\hat{x}_i\hat{x}_j\hat{x}_k^{-1}\hat{x}_j^{-1}$ or $\hat{x}_j\hat{x}_i\hat{x}_j^{-1}\hat{x}_k^{-1}$ depending on whether the crossing point corresponding to $\hat{r}_i$ is positive or negative.
The pair of relators $(r_i, r_i^\ast)$ in $D$ and $D^\ast$ are also of the form
$(r_i, r_i^\ast)=(x_i x_j x_k^{-1} x_j^{-1}, x_i^\ast x_j^\ast (x_k^\ast)^{-1} (x_j^\ast)^{-1})$ or $(x_j x_i x_j^{-1} x_k^{-1}, x_j^\ast x_i^\ast (x_j^\ast)^{-1} (x_k^\ast)^{-1})$ depending on whether the crossing point corresponding to $r_i$ is positive or negative as shown in Figure \ref{pfmainthm2_relator_01}.
The relators $r_k^i$ in $T_i$ are of the form $r_k^i=s_i^i s_j^i (s_k^i)^{-1} (s_j^i)^{-1}$ or $s_j^i s_i^i (s_j^i)^{-1} (s_k^i)^{-1}$.

\begin{figure}[h]
	\centerline{\includegraphics[keepaspectratio]{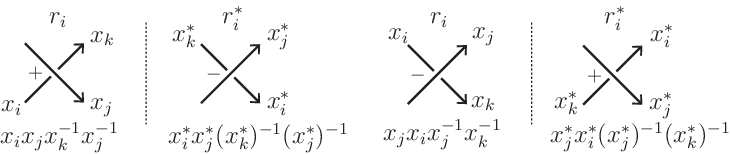}}
\caption{$(r_i, r_i^\ast)=(\text{positive},\text{negative})$ or $(\text{negative},\text{positive})$}
\label{pfmainthm2_relator_01}
\end{figure}

From the observation above, we see that the map $\varphi$ maps each relator to relators of $G(\hat{K})$ as follows.
\begin{align*}
\varphi(r_i)=\varphi(r_i^\ast)
&=\hat{x}_i \hat{x}_j \hat{x}_k^{-1} \hat{x}_j^{-1} \;\;\text{or}\;\; \hat{x}_j \hat{x}_i \hat{x}_j^{-1} \hat{x}_k^{-1} = \hat{r}_i \\
\varphi(r_k^i)
&=\tilde{x}_i\tilde{x}_i\tilde{x}_i^{-1}\tilde{x}_i^{-1}=1, \;\text{for}\; k=1,2,\dots, n_i
\end{align*}

Therefore, the map $\varphi$ is a homomorphism.
Furthermore, since all the generators of $G(\hat{K})$ are images of the generators of $G(K)$, the map $\varphi$ is surjective.

Based on the construction above,
the word $\gamma$ which represents the longitude $\lambda_K$ of $K$ has the following form:
\begin{align*}
\gamma &=\gamma_-\cdot\gamma_+ \\
&= w_m \cdots w_{i_n+1} S_-^{k_n} w_{i_n} \cdots w_{i_j+1} S_-^{k_j} w_{i_j} \cdots w_{i_1+1} S_-^{k_1} w_{i_1} \cdots w_2 w_1 \\
& \phantom{= \ell_-} \cdot w_1^\ast w_2^\ast \cdots w_{i_1}^\ast S_+^{k_1} w_{i_1+1}^\ast \cdots w_{i_j}^\ast S_+^{k_j} w_{i_j+1}^\ast \cdots w_{i_n}^\ast S_+^{k_n} w_{i_n+1} \cdots w_m^\ast
\end{align*}

The sub-word $w_1 w_1^\ast$ located at the center of $\gamma$ has the form
$$w_1w_1^\ast=
x_{i_1}^{-1}y_{i_1}(y_{i_1}^\ast)^{-1}x_{i_1}^\ast \;\;\text{or}\;\; x_{i_1}y_{i_1}^{-1}y_{i_1}^\ast (x_{i_1}^\ast)^{-1}
$$
then $\varphi(w_1w_1^\ast)=1$.
Similarly, we obtain
$\varphi(w_iw_i^\ast)=1$ successively for each index $i$,
and $\varphi(S_-^{k_i})=\varphi(S^{k_i}_+)=(\text{products of}\;(\tilde{x}_i\tilde{x}_i^{-1})\;\text{or}\;(\tilde{x}_i^{-1}\tilde{x}_i))=1$ for all index $i$.

By the above arguments, we see that $\varphi(\lambda_K)=\varphi(\gamma)=1$,
then the proof of Theorem \ref{maintheorem}(2) is completed.

\section{Examples}
We will present several examples based on the main theorem and examine their properties.

Kinoshita and Terasaka introduce an infinite family of knots with an Alexander polynomial equal to $1$ \cite{KT1957}.
Let $T_n^{KT}$ be a tangle with two twist regions that are used for the partial knots and also for the inserted parts, having $n$ and $n-1$ twists, as shown in the middle of Figure \ref{example1_01} for $n\geq 3$.

\begin{figure}[h]
	\centerline{\includegraphics[keepaspectratio]{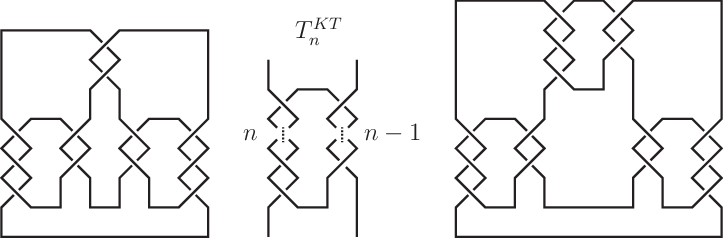}}
\caption{Kinoshita-Terasaka knot, tangle $T_n^{KT}$, and $K_A$}
\label{example1_01}
\end{figure}

Knots in this infinite family of knots are regarded as symmetric unions of the partial knot $N(T_n^{KT})$ with an even twist region.
Note that the numerator $N(T_n^{KT})$ is a trivial knot.
The knot on the left of Figure \ref{example1_01} is the well-known ``Kinoshita-Terasaka knot'', which corresponds to the case $n=3$ and is listed as 11n42 in the Hoste-Thistlethwaite table.
From the perspective of the main theorem, this knot is regarded as in the case where the partial knot $\hat{K}$ is $N(T_3^{KT})$ and $T_1$ is the region with a full twist.

Replacing the twist region of the Kinoshita-Terasaka knot with the tangle $T_3^{KT}$, we obtain the knot $K_A$ on the right of Figure \ref{example1_01}.
We observe that the knot $K_A$ is equivalent to 13n3934, its Alexander polynomial $\Delta_{K_A}(t)$ is trivial, and its Jones polynomial is equal to $V_{K_A}(t)=t^{-5}-3t^{-4}+5t^{-3}-6t^{-2}+5t^{-1}-2+3t^2-5t^3+6t^4-5q^5+3t^6-t^7$
by using SnapPy \cite{SnapPy}.

The knot $K_A$ has a single tangle region, so this type of construction has already been discussed in \cite{KN2025}.
By adding the tangle $T_3^{KT}$ at the position of the twice twisted part of $T_3^{KT}$ and $(T_3^{KT})^\ast$ in $K_A$, we obtain a knot $K_B$ on the left of Figure \ref{example1_02}.
We can also verify with SnapPy that the Alexander polynomial is trivial $\Delta_{K_B}(t)=1$, and the Jones polynomial is
$V_{K_B}=t^{-6} - 3t^{-5} + 5t^{-4} - 5t^{-3} + t^{-2} + 5t^{-1} - 10 + 14t - 13t^2 + 9t^3 - 2t^4 - 5t^5 + 10t^6 - 11t^7 + 8t^8 - 4t^9 + t^{10}$.

By considering $N(T_4^{KT})$ as a partial knot and attaching three tangles $T_3^{KT}$,
we obtain a knot $K_C$, as shown to the right of Figure \ref{example1_02}.
Using SnapPy for the calculation, the Alexander polynomial of $K_C$ is again trivial $\Delta_{K_C}(t)=1$, and the Jones polynomial is $V_{K_C}(t)=t^{-9} - 4t^{-8} + 10t^{-7} - 16t^{-6} + 15t^{-5} - 31t^{-3} + 66t^{-2} - 87t^{-1} + 77 - 29t - 40t^2 + 108t^3 - 148t^4 + 146t^5 - 102t^6 + 33t^7 + 34t^8 - 77t^9 + 86t^{10} - 69t^{11} + 42t^{12} - 19t^{13} + 6t^{14} - t^{15}$.

As shown in these examples, we can easily create knots that have a trivial Alexander polynomial and a complicated Jones polynomial using our construction.

\begin{figure}[h]
	\centerline{\includegraphics[keepaspectratio]{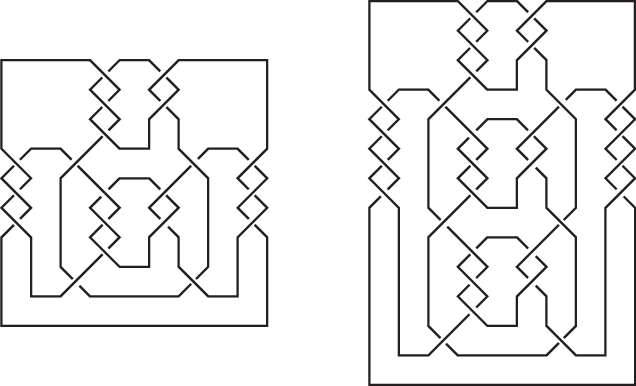}}
\caption{$K_B$ and $K_C$}
\label{example1_02}
\end{figure}

Kitano and Suzuki provided a complete list of knot pairs where $K_1\geq K_2$ up to 10 crossings \cite{KS2005}, and these maps that realize epimorphisms for these pairs are explicitly described in \cite{KS2008}.
Based on their list, we examined all pairs $K_1 \geq K_2$ where the epimorphisms map the longitude of $G(K_1)$ to the trivial element of  $G(K_2)$ when $K_2$ equals $3_1$ or $4_1$, as in \cite{KN2025}.
Except for the two knots $10_{82}$ and $10_{87}$, all of these knots can be constructed using our construction in the case of a single tangle region.

Horie, Kitano, Matsumoto, and Suzuki provided a complete list of pairs $K_1 \geq K_2$ for cases where $K_1$ has 11 crossings and $K_2=3_1, 4_1, 5_1, 5_2, 6_1, 6_2, 6_3$ \cite{HKMS2011}.
The three knots 11a157, 11a264, and 11a305 are the only knots on the list for the case $K_2=3_1$ where the Alexander polynomial is divisible by the square of the Alexander polynomial of the trefoil and their quotients are products of two polynomials.
The knot 11a157 can be constructed using our construction for the case of two tangle regions, as shown on the left of Figure \ref{example2_01}, which is denoted by $K_D$.
The knot $K_D'$ on the right of Figure \ref{example2_01} is formed by taking a single tangle region as the connected sum of the trefoil and the figure-eight knot.
However, the knot $K_D'$ is also recognized as 11a157 by SnapPy.
Although the Alexander polynomials of the other knots 11a264 and 11a305 are the same as those of 11a157,
we have not yet succeeded in constructing these knots as knots built using our construction. 

\begin{figure}[h]
	\centerline{\includegraphics[keepaspectratio]{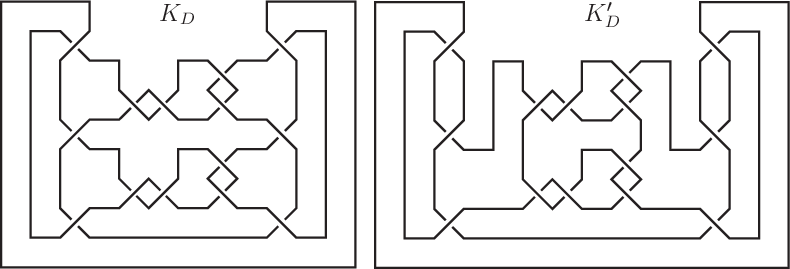}}
\caption{11a157: $K_D$ and $K_D'$}
\label{example2_01}
\end{figure}

Next, we examine the knots in the list that have an epimorphism to the knot group of the figure-eight knot.
The knots in the list whose Alexander polynomials are divisible by the square of the Alexander polynomial of the figure-eight knot are 11a5 and 11a297.
The Alexander polynomials of these knots are as follows.
\begin{align*}
\Delta_{11a5}(t)&=(t^2-3t+1)^3=\Delta_{4_1}(t)\cdot\left(\Delta_{4_1}(t)\right)^2 \\
\Delta_{11a297}(t)&=(2t^2-3t+2)(t^2-3t+1)^2=\Delta_{5_2}(t)\cdot\left(\Delta_{4_1}(t)\right)^2
\end{align*}
The knot 11a5 is constructed as a knot with a single tangle region using our construction as shown in Figure \ref{example2_02}, which is denoted by $K_E$.
We also have not yet succeeded in constructing 11a297 as a knot built using our construction.

\begin{figure}[h]
	\centerline{\includegraphics[keepaspectratio]{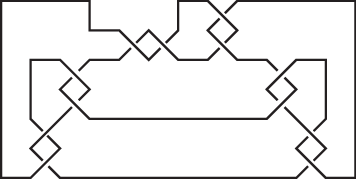}}
\caption{11a5: $K_E$}
\label{example2_02}
\end{figure}

Following the considerations above, we propose the following problems.

\begin{prob}
Determine whether the knots 11a264 and 11a305 can be constructed using the method of Theorem \ref{maintheorem} in the case that the trefoil is a partial knot.
Also, determine this for 11a297 in the case that the figure-eight knot is a partial knot.
\end{prob}

\vspace{8pt}

\noindent\textbf{Acknowledgements:}
The authors would like to express their gratitude to Christoph Lamm for valuable comments on their previous paper,
for inspiring them to generalize their theorem to the case with multiple tangles,
and for his helpful comments on the draft of this paper.
The first author is partially supported by JSPS KAKENHI Grant Number 19K03505.
The second author is partially supported by JSPS KAKENHI Grant Number 19K03460.

\end{document}